\documentclass[12pt]{article}%
\usepackage{amsfonts}
\usepackage{amsmath}
\usepackage{geometry}
\usepackage{amssymb}
\usepackage{graphicx}%
\newtheorem{theorem}{Theorem}

\newtheorem{remark}{Remark}

\newenvironment{proof}[1][Proof]{\noindent\textbf{#1.} }{\ \rule{0.5em}{0.5em}}
\begin{document}

\title{An Elementary Approach to a Model Problem of Lagerstrom}
\author{S. P. Hastings and J. B. McLeod}
\maketitle

\begin{abstract}
The equation studied is $u^{\prime\prime}+\frac{n-1}{r}u^{\prime}+\varepsilon
u~u^{\prime}+ku^{\prime2}=0,$ \ with boundary conditions $u\left(  1\right)
=0,$ $u\left(  \infty\right)  =1$. \ \ This model equation has been studied by
many authors since it was introduced in the 1950s by P. A. Lagerstrom. We use
an elementary approach to show that there is an infinite series solution which
is uniformly convergent on $1\leq r<\infty.$ The first few terms are easily
derived, from which one quickly deduces the inner and outer asymptotic
expansions, with no matching procedure or a priori assumptions about the
nature of the expansion. \ We also give a short and elementary existence and
uniqueness proof which covers all $\varepsilon>0,$ $k\geq0,$ and $n\geq1$.

\end{abstract}

\section{Introduction}

The main problem is to investigate the asymptotics as $\varepsilon
\rightarrow0$ of the boundary value problem
\begin{equation}
u^{\prime\prime}+\frac{n-1}{r}u^{\prime}+\varepsilon uu^{\prime}+ku^{\prime
2}=0 \label{1}%
\end{equation}
with
\begin{equation}
u\left(  1\right)  =0,u\left(  \infty\right)  =1. \label{2}%
\end{equation}

We consider the cases $k=0$ and $k=1.$ Our interest in these problems,
originally due to Lagerstrom in the 1950s \cite{kla},\cite{la}, \ was
stimulated by two recent papers by Popovic and Szmolyan \cite{pop1}%
,\cite{pop2}, who adopt a geometric approach to the problem when $k=0$, and
there are many papers which use methods of matched asymptotics or multiple
scales, with varying degrees of rigor. \ We will review some of this work
below. \ The point of this paper is to give a completely rigorous and
relatively short answer to the problem without making any appeal either to
geometric methods or to matched asymptotics. \ We can express the solution as
an infinite series, uniformly convergent for all values of the independent
variable. \ From this series we obtain the inner and outer asymptotic
expansions with no a priori assumption about the nature of these expansions.
\ An important, and as far as we know, original, feature is that there is no
\textquotedblleft matching\textquotedblright.

\bigskip

Lagerstrom came up with these problems as models of viscous incompressible
($k=0$) and compressible ($k=1$) flow, so much of his work centered on $n=2$
or $3$, but he also discussed general $n\geq1$ \cite{lagerrein}. \ The
infinite series we develop can be obtained for any real number $n$. \ What $n$
controls is the rate of convergence of the series. \ 

\bigskip

For $\varepsilon=k=0,$ there is an obvious distinction between $n>2$ and
$n\leq2$. \ If $n>2$, \ then the problem $\left(  \ref{1}\right)  -\left(
\ref{2}\right)  $ has the unique solution
\begin{equation}
u=1-\frac{1}{r^{n-2}}, \label{1.1}%
\end{equation}
so that the solution with $\varepsilon$ \ small is presumably some sort of
perturbation of this. If $n\leq2$ \ then there is no such solution. \ \ A
consequence is that the convergence as $\varepsilon\rightarrow0$ \ is more
subtle when $n\leq2$ \ then when $n>2$. \ Our analysis will show that there is
little prospect of discussing the behavior for small $\varepsilon$ if $n<2,$
but fortunately we can handle all $n\geq2$. \ \ \ Although it has been thought
that finding the asymptotics when $k=1$ is considerably more difficult than
when $k=0,$ \cite{hi}, we will show that our technique covers each case with
comparable effort.

\bigskip

Our methods are not restricted to Lagerstrom's problems $\left(
\ref{1}\right)  -\left(  \ref{2}\right)  $. \ In subsequent work (in
preparation), we will show that there is a general method which can yield
similar results for a class of singularly perturbed boundary value problems.

\bigskip

We start in section \ref{sec2} by showing that each of these problems has one
and only one solution, for any $n\geq1$ and any $\varepsilon>0$. \ \ This is
based on a simple shooting argument plus a comparison principle. \ These
results have been obtained before, but our proof is quite short. \ \ In the
subsequent sections we develop the integral equation referred to above, and
show how it leads with relative ease to the inner and outer expansions.
\ These expansions go back to Lagerstrom and Kaplun, with rigorous
justification of some of the features to be found in \cite{cohen} or
\cite{pop2}, for example. \ We find the exposition in Hinch's book \cite{hi}
particularly clear (though nonrigorous), and make that our point of comparison
in checking that we get the same expansions as were found
previously.\footnote{We thank the referees for some very helpful comments.
\ In particular, they called our attention to earlier proofs of the existence
and uniqueness results, in some cases by techniques similar to ours.}

\bigskip

\section{Existence and uniqueness\label{sec2}}

\bigskip

As far as we know, the first existence proof was by Hsiao \cite{hsiao}, who
only considered $n=1$ and sufficiently small $\varepsilon>0.$ Subsequently Tam
gave what seems to be the first proof valid for all $\varepsilon>0$ and
$k\geq0$, \cite{tam}. \ Subsequent proofs by MacGillivray \cite{macgillivray},
Cohen, Lagerstrom and Fokas \cite{cohen}, Hunter, Tajdari, and Boyer
\cite{hunter}, each of which covers all $\varepsilon>0,$ and by several other
authors, e.g. \cite{rosenblat},\cite{pop1}, for restricted ranges of
$\varepsilon,$ add to the variety of techniques which have been shown to work.
\ Uniqueness is proved in \cite{hunter} (for $k=0$) by use of a contraction
mapping theorem, and in \cite{cohen} by essentially a comparison method. \ The
goal of \cite{pop1} is not to give a short proof, but to illustrate the
application of geometric perturbation theory to a much studied problem in
matched asymptotic expansions. \ The proofs we give of existence and
uniqueness are considerably shorter than the others we have seen.

\begin{theorem}
\label{th1}\ There exists a unique solution to the problem $\left(
\ref{1}\right)  -\left(  \ref{2}\right)  $ for any $k\geq0,$ $\varepsilon>0,$
and $n\geq1.$
\end{theorem}

\begin{proof}
Like some others, starting with \cite{tam}, we prove existence using a
shooting method, by considering the initial value problem
\begin{align}
u^{\prime\prime}+\frac{n-1}{r}u^{\prime}+\varepsilon uu^{\prime}+ku^{\prime2}
&  =0\label{e1}\\
u\left(  1\right)  =0,~u^{\prime}\left(  1\right)  =  &  c, \label{e1a}%
\end{align}
for each $c>0.$ Since $u^{\prime}=0$ implies that $u^{\prime\prime}=0$ and $u$
is constant, any solution to this problem is positive and increasing. \ As was
observed in \cite{tam},
\[
u^{\prime\prime}+\varepsilon uu^{\prime}\leq0,
\]
and so from $\left(  \ref{e1a}\right)  ,$%
\[
u^{\prime}+\frac{1}{2}\varepsilon u^{2}\leq c.
\]
In particular, since $u^{\prime}\geq0,$
\begin{equation}
u\leq\sqrt{\frac{2c}{\varepsilon}}, \label{e1b}%
\end{equation}
so the solution exists, and satisfies this bound, on $[1,\infty)$.
\ Therefore, $\lim_{r\rightarrow\infty}u\left(  r\right)  $ exists. \ Writing
the equation in the form
\begin{equation}
\left(  r^{n-1}u^{\prime}\right)  ^{\prime}+\left(  \varepsilon u+ku^{\prime
}\right)  \left(  r^{n-1}u^{\prime}\right)  =0, \label{e2}%
\end{equation}
and integrating twice, gives
\begin{align}
r^{n-1}u^{\prime}\left(  r\right)   &  =ce^{-ku\left(  r\right)
-\varepsilon\int_{1}^{r}u\left(  s\right)  ds},\label{e4}\\
u\left(  r\right)   &  =\int_{1}^{r}\frac{c}{s^{n-1}}e^{-ku-\int_{1}%
^{s}\varepsilon udt}ds. \label{e4a}%
\end{align}
If $u\left(  2\right)  <1,$ then since $u$ is increasing, $\left(
\mathbf{\ref{e4a}}\right)  $ implies that
\[
u\left(  2\right)  >p\left(  c\right)  =\int_{1}^{2}\frac{c}{s^{n-1}%
}e^{-\varepsilon-k}ds.
\]
From this, and $\left(  \ref{e1b}\right)  ,$ we see that there are $c_{1}$ and
$c_{2},$ with $0<c_{1}<c_{2}$, such that if $c=c_{1}$ then $u\left(
\infty\right)  <1,$ while if $c=c_{2},$ then $u\left(  \infty\right)
>u\left(  2\right)  \geq1$. \ Further, from $\left(  \ref{e4}\right)  $ for
any $r>R>2,$ \
\[
u\left(  r\right)  =u\left(  R\right)  +c\int_{R}^{r}\frac{1}{s^{n-1}%
}e^{-ku-\int_{1}^{s}\varepsilon udt}ds.
\]
If $n>2$ \ the second term is bounded above by $\frac{c}{\left(  n-2\right)
R^{n-2}},$ while if $1\leq n\leq2$ and $R\geq2,$ it is bounded by $c\int
_{R}^{\infty}e^{-\varepsilon\left(  s-2\right)  p\left(  c\right)  }ds.$
Hence, this term tends to zero as $R\rightarrow\infty,$ uniformly for $r\geq
R,$ $c_{1}\leq c\leq c_{2}$. \ Since $u\left(  R\right)  $ is a continuous
function of $c,$ for any $R,$ it follows that $u\left(  \infty\right)  $ is
also continuous in $c$, and so there is a $c$ with $u\left(  \infty\right)
=1,$ giving a solution to $\left(  \ref{1}\right)  -\left(  \ref{2}\right)  .$

For uniqueness, suppose that there are two solutions of $\left(
\ref{1}\right)  -\left(  \ref{2}\right)  ,$ say $u_{1}$ and $u_{2},$\ with
$u_{1}^{\prime}\left(  1\right)  >u_{2}^{\prime}\left(  1\right)  >0.$ Then
$u_{1}>u_{2}$ on some maximal interval, say $\left(  1,X\right)  $ where
$X\leq\infty.$ \ For the same initial conditions, if $\varepsilon=k=0$ then
direct integration shows that $u_{1}>u_{2}$ on $\left(  1,\infty\right)  $,
and moreover, $u_{1}\left(  \infty\right)  >u_{2}\left(  \infty\right)  $.
\ We then raise $\varepsilon$ and $k,$ looking for a pair $\left(
\varepsilon_{1},k_{1}\right)  $ such that $u_{1}\left(  X\right)
=u_{2}\left(  X\right)  $ for some $X\leq\infty,$ and if $0\leq\varepsilon
<\varepsilon_{1}$ or $0\leq k<k_{1},$ no such $X$ exists. \ Hence, at $\left(
\varepsilon_{1},k_{1}\right)  ,$ $u_{1}\geq u_{2}$ on $[0,\infty)$. \ If, at
$\left(  \varepsilon_{1},k_{1}\right)  ,$ $X<\infty$, \ then $u_{1}$ and
$u_{2}$ must be tangent at $X,$ since $u_{1}-u_{2}$ has a minimum there,
\ contradicting the uniqueness of initial value problems for $\left(
\ref{1}\right)  $. \ Hence, $X=\infty$, and $u_{1}>u_{2}$ on $\left(
1,\infty\right)  $.

Observe from $\left(  \ref{e2}\right)  $ that if $u_{1}^{\prime}\left(
r\right)  =u_{2}^{\prime}\left(  r\right)  $ for some $r,$ then $\left(
r^{n-1}\left(  u_{1}^{\prime}-u_{2}^{\prime}\right)  \right)  ^{\prime}<0,$
since $u_{1}>u_{2},$ so that there cannot be oscillations in $u_{1}^{\prime
}-u_{2}^{\prime}$. Hence, $u_{1}\left(  \infty\right)  =u_{2}\left(
\infty\right)  $ implies that there is an $R$ with $u_{1}^{\prime}\left(
R\right)  =u_{2}^{\prime}\left(  R\right)  $ and $u_{1}^{\prime}<u_{2}%
^{\prime}$ on $\left(  R,\infty\right)  $. \ Integrating $\left(
\ref{e2}\right)  ,$ and recalling that $u_{1}\left(  \infty\right)
=u_{2}\left(  \infty\right)  ,$ gives
\begin{align*}
r^{n-1}\left(  u_{1}^{\prime}-u_{2}^{\prime}\right)  |_{R}^{\infty}  &
=\frac{1}{2}\varepsilon R^{n-1}\left(  u_{1}^{2}-u_{2}^{2}\right)  |_{R}\\
&  +\frac{1}{2}\varepsilon\left(  n-1\right)  \int_{R}^{\infty}s^{n-2}\left(
u_{1}^{2}-u_{2}^{2}\right)  ds-k\int_{R}^{\infty}s^{n-1}\left(  u_{1}%
^{\prime2}-u_{2}^{\prime2}\right)  ds.
\end{align*}
The left hand side is zero, and all the terms on the right are positive,
giving the necessary contradiction.
\end{proof}

\begin{remark}
The existence theorem in \cite{pop1} has one added part. \ It is shown there
that as $\varepsilon\rightarrow0,$\ \ the solution tends to a so-called
\textquotedblleft singular\textquotedblright\ solution obtained by taking a
formal limit as $\varepsilon\rightarrow0.$\ See \cite{pop1} \ for details.
\ This limit result follows from our rigorous asymptotic expansions given below.
\end{remark}

\begin{remark}
There would seem to be no difficulty in extending the existence proof even to
$n<1,$\ but the uniqueness proof does use essentially the fact that $n\geq1.$\ 
\end{remark}

\section{The infinite series (with $k=0,$ $n\geq2$)\label{infser}}

\bigskip

Starting again with $\left(  \ref{1}\right)  $, \ and $u\left(  1\right)  =0,$
\ we first consider the case $k=0,$ and obtain
\begin{equation}
r^{n-1}u^{\prime}=Be^{-\varepsilon\int_{1}^{r}u\left(  t\right)  dt}
\label{3.1}%
\end{equation}
for some constant $B$. \ Since $u\left(  \infty\right)  =1$, \ $\left(
\ref{3.1}\right)  $ implies that $u^{\prime}\left(  r\right)  $ is
exponentially small as $r\rightarrow\infty$. \ \ Hence we can rewrite $\left(
\ref{3.1}\right)  $ as
\[
r^{n-1}u^{\prime}=Ce^{-\varepsilon r-\varepsilon\int_{\infty}^{r}\left(
u-1\right)  dt},
\]
so that
\[
u-1=C\int_{\infty}^{r}\frac{1}{t^{n-1}}e^{-\varepsilon t-\varepsilon
\int_{\infty}^{t}\left(  u-1\right)  ds}dt.
\]
Setting $\varepsilon r=\rho$, $\varepsilon t=\tau$, $\ $and $\varepsilon
s=\sigma,$ \ we obtain
\begin{equation}
u\left(  \rho\right)  -1=C\varepsilon^{n-2}\int_{\infty}^{\rho}\frac{1}%
{\tau^{n-1}}e^{-\tau}e^{-\int_{\infty}^{\tau}\left(  u\left(  \sigma\right)
-1\right)  d\sigma}d\tau, \label{3.2}%
\end{equation}
where we use the arguments $\rho$ and $\sigma$ to indicate that we mean the
rescaled version of $u$. \ \ Here $C$ is a constant satisfying%
\begin{equation}
-1=C\varepsilon^{n-2}\int_{\infty}^{\varepsilon}\frac{1}{\tau^{n-1}}e^{-\tau
}e^{-\int_{\infty}^{\tau}\left(  u-1\right)  d\sigma}d\tau. \label{3.2a}%
\end{equation}
Since for each $\varepsilon$ there is a unique solution, this determines a
unique $C,$ dependent on $\varepsilon$. \ 

\bigskip

We now consider the integral
\begin{equation}
\int_{\tau}^{\infty}\left(  1-u\left(  \sigma\right)  \right)  d\sigma,
\label{3.2b}%
\end{equation}
which appears in the exponent in $\left(  \ref{3.2a}\right)  $. \ This
integral has been seen to converge for each $\varepsilon$, but we need a bit
more, namely, that it is bounded uniformly in $\varepsilon\leq\tau<\infty$
\ as $\varepsilon\rightarrow0.$ To see this, we note that as a function of
$\sigma,$ $u$ satisfies
\begin{align*}
\frac{d^{2}u}{d\sigma^{2}}+\frac{n-1}{\sigma}\frac{du}{d\sigma}+u\frac
{du}{d\sigma}  &  =0\\
u=0\text{ when }\sigma &  =\varepsilon\text{, }u\left(  \infty\right)  =1.
\end{align*}
Denoting the unique solution by $u_{\varepsilon}\left(  \sigma\right)  ,$ we
claim that if $0<\varepsilon_{1}<\varepsilon_{2},$ then $u_{\varepsilon_{1}%
}>u_{\varepsilon_{2}}$ for $\varepsilon_{2}\leq\sigma<\infty$. \ \ If this is
false, then $\varepsilon_{1}$ and $\varepsilon_{2}$ \ can be chosen so that
$u_{\varepsilon_{1}}\left(  \sigma_{0}\right)  =u_{\varepsilon_{2}}\left(
\sigma_{0}\right)  $ for some $\sigma_{0}\geq\varepsilon_{2}$. \ But then, the
problem
\begin{align*}
\frac{d^{2}u}{d\sigma^{2}}+\frac{n-1}{\sigma}\frac{du}{d\sigma}+u\frac
{du}{d\sigma}  &  =0\\
u\left(  \sigma_{0}\right)  =u_{\varepsilon_{1}}\left(  \sigma_{0}\right)
,~u\left(  \infty\right)   &  =1
\end{align*}
has two solutions, contradicting our earlier uniqueness proof. \ 

A consequence of this is that $\int_{\varepsilon_{2}}^{\infty}\left(
1-u_{\varepsilon_{1}}\left(  \sigma\right)  \right)  d\sigma<\int
_{\varepsilon_{2}}^{\infty}\left(  1-u_{\varepsilon_{2}}\left(  \sigma\right)
\right)  d\sigma$, \ which implies that the integral in the exponent in
$\left(  \ref{3.2a}\right)  ,$ including the minus sign in front, is bounded
below independently of $\tau\geq\varepsilon,$ \ and of $\varepsilon$. \ We
then see that the $\tau$-integral in $\left(  \ref{3.2a}\right)  $ approaches
$-\infty$ as $\varepsilon\rightarrow0,$ and hence, that%

\[
\lim_{\varepsilon\rightarrow0^{+}}C\varepsilon^{n-2}=0.
\]

Since $\ \int_{\infty}^{\tau}\left(  u-1\right)  d\sigma>0,$ it follows from
$\left(  \ref{3.2}\right)  $ that if
\[
E_{n-1}\left(  \rho\right)  =\int_{\rho}^{\infty}\frac{1}{\tau^{n-1}}e^{-\tau
}d\tau,
\]
then
\begin{equation}
\left\vert u\left(  \rho\right)  -1\right\vert <C\varepsilon^{n-2}%
E_{n-1}\left(  \rho\right)  . \label{3.3}%
\end{equation}
\bigskip\ \ \ 

For purposes of future estimates, we make the obvious remark that
\begin{equation}
E_{n-1}\left(  \rho\right)  =\left\{
\begin{array}
[c]{c}%
O\left(  \rho^{2-n}\right)  \text{ as }\rho\rightarrow0\text{ if }n>2\\
O\left(  \log\rho\right)  \text{ as }\rho\rightarrow0\text{ if }n=2\\
O\left(  \rho^{1-n}e^{-\rho}\right)  \text{ as }\rho\rightarrow\infty.
\end{array}
\right.  \label{3.4}%
\end{equation}
Hence if $n>2$ \ there is a constant $K$ such that
\begin{equation}
E_{n-1}\left(  \rho\right)  \leq K\min\left(  \rho^{2-n},\rho^{1-n}e^{-\rho
}\right)  . \label{3.3b}%
\end{equation}

The method now is to work from $\left(  \ref{3.2}\right)  $. \ As observed
before, since $u^{\prime}\left(  r\right)  $ is exponentially small as
$r\rightarrow\infty$, \ \ the integral term $\int_{\rho}^{\infty}\left(
u-1\right)  d\sigma$ converges.\ \ \ Hence, for given $\varepsilon>0$ and
$\rho_{0}>0,$ and any $\rho\geq\rho_{0}$,%
\begin{equation}
u\left(  \rho\right)  -1=C\varepsilon^{n-2}\int_{\infty}^{\rho}\frac{1}%
{\tau^{n-1}}e^{-\tau}\left\{  1-\int_{\infty}^{\tau}\left(  u-1\right)
d\sigma+\frac{1}{2}\left(  \int_{\infty}^{\tau}\left(  u-1\right)
d\sigma\right)  ^{2}-\cdot\cdot\cdot\right\}  d\tau, \label{3.3a}%
\end{equation}
where the series in the integrand converges uniformly for $\rho_{0}\leq
\tau<\infty$. \ \ 

\bigskip

In fact, we will need to use this series for all $\rho\geq\varepsilon$. \ Thus
we need to check its convergence in this interval. This follows from $\left(
\ref{3.3}\right)  $ and $\left(  \ref{3.4}\right)  ,$ which imply that for any
$\rho\geq\varepsilon,$ if $n\geq2,$ \ then
\begin{equation}
\left\vert \int_{\rho}^{\infty}\left(  u\left(  s\right)  -1\right)
ds\right\vert <C\varepsilon^{n-2}\int_{\varepsilon}^{\infty}E_{n-1}\left(
s\right)  ds \label{3.3c}%
\end{equation}
and
\[
\varepsilon^{n-2}\int_{\varepsilon}^{\infty}E_{n-1}\left(  s\right)
ds=\left\{
\begin{array}
[c]{c}%
o\left(  1\right)  \text{ as }\varepsilon\rightarrow0\text{ if }n>2\\
O\left(  1\right)  \text{ as }\varepsilon\rightarrow0\text{ if }n=2
\end{array}
.\right.
\]
$\newline$Hence for $n>2$ and any $C,$ the series in the integrand of $\left(
\ref{3.3a}\right)  $ converges uniformly on $[\varepsilon,\infty).$

Now set
\[
\Phi=C\varepsilon^{n-2}\int_{\varepsilon}^{\infty}E_{n-1}\left(  s\right)
ds.
\]
We note that, if $n>2,$ then $\Phi\rightarrow0$ as $\varepsilon\rightarrow0,$
while if $n=2,$ then $\Phi\rightarrow0$ as $C\rightarrow0$.

We proceed to solve $\left(  \ref{3.3a}\right)  $ by iteration. \ Thus, the
first approximation is, from $\left(  \ref{3.3b}\right)  ,$%
\[
u\left(  \rho\right)  -1=C\varepsilon^{n-2}\int_{\infty}^{\rho}\frac{1}%
{\tau^{n-1}}e^{-\tau}d\tau+O\left(  \Phi^{2}\right)  ,
\]
and we obtain the second approximation by substituting this back in $\left(
\ref{3.3a}\right)  .$ \ Repeating this, we reach%

\begin{equation}%
\begin{array}
[c]{c}%
u-1=-C\varepsilon^{n-2}E_{n-1}+\left(  C\varepsilon^{n-2}\right)  ^{2}%
\int_{\infty}^{\rho}\frac{1}{\tau^{n-1}}e^{-\tau}\left(  \int_{\infty}^{\tau
}E_{n-1}d\sigma\right)  d\tau\\
+\frac{1}{2}\left(  C\varepsilon^{n-2}\right)  ^{3}\int_{\infty}^{\rho}%
\frac{1}{\tau^{n-1}}e^{-\tau}\left(  \int_{\infty}^{\tau}E_{n-1}%
d\sigma\right)  ^{2}d\tau\\
-\left(  C\varepsilon^{n-2}\right)  ^{3}\int_{\infty}^{\rho}\frac{1}%
{\tau^{n-1}}e^{-\tau}\int_{\infty}^{\tau}\left\{  \int_{\infty}^{\sigma}%
\frac{1}{s^{n-1}}e^{-s}\left(  \int_{\infty}^{s}E_{n-1}dt\right)  ds\right\}
d\sigma d\tau+O\left(  \Phi^{4}\right)  ,
\end{array}
\label{3.9}%
\end{equation}
as $\Phi\rightarrow0$. \ \ 

To obtain $C,$ we need to be able to evaluate each of these terms for small
$\rho$ (in particular, for $\rho=\varepsilon$), \ and this is a matter of
integration by parts. \ Thus, for non-integral $n$,
\begin{align}
E_{n-1}\left(  \rho\right)   &  =\int_{\rho}^{\infty}\frac{e^{-\tau}}%
{\tau^{n-1}}d\tau=-\frac{\rho^{2-n}}{2-n}e^{-\rho}+\frac{1}{2-n}\int_{\rho
}^{\infty}\frac{e^{-\tau}}{\tau^{n-2}}d\tau\nonumber\\
&  =-\frac{\rho^{2-n}}{2-n}e^{-\rho}+\frac{1}{2-n}E_{n-2}, \label{3.9a}%
\end{align}
and this can be repeated to give $E_{n-1}$ \ as a sum of terms of the form
$c_{k}\rho^{k}e^{-\rho}$ and $E_{n-p},$ until $0<n-p<1$.\ \ Then%
\begin{align*}
E_{n-p}  &  =\int_{0}^{\infty}\frac{e^{-\tau}}{\tau^{n-p}}d\tau-\int_{0}%
^{\rho}\frac{e^{-\tau}}{\tau^{n-p}}d\tau\\
&  =\Gamma\left(  p+1-n\right)  -\int_{0}^{\rho}\frac{e^{-\tau}}{\tau^{n-p}%
}d\tau
\end{align*}
and we can then continue to integrate by parts as far as we like. \ (If $n$ is
an integer, we will reach $\int_{\rho}^{\infty}\frac{e^{-\tau}}{\tau}d\tau$,
\ which introduces a logarithm.)

\bigskip

Thus $E_{n-1}\left(  \rho\right)  $ can be expressed as a sum of terms of the
form $c_{k}\rho^{k}e^{-\rho}$, \ and so obviously the same is true of
$E_{n-1}^{2}$ , with $e^{-2\rho}$ in place of $e^{-\rho}$. \ Also,
\begin{align}
\int_{\infty}^{\rho}E_{n-1}\left(  \tau\right)  d\tau &  =\int_{\infty}^{\rho
}\left(  \int_{\tau}^{\infty}\frac{e^{-\sigma}}{\sigma^{n-1}}d\sigma\right)
d\tau\nonumber\\
&  =\left[  \tau\left(  \int_{\tau}^{\infty}\frac{e^{-\sigma}}{\sigma^{n-1}%
}d\sigma\right)  \right]  |_{\infty}^{\rho}+\int_{\infty}^{\rho}\frac
{e^{-\tau}}{\tau^{n-2}}d\tau\nonumber\\
&  =\rho E_{n-1}-E_{n-2}, \label{3.10}%
\end{align}
so that $\int_{\infty}^{\rho}E_{n-1}d\tau$ \ can be expressed as the same type
of sum. \ Hence the second term in $\left(  \ref{3.9}\right)  $ gives a sum of
terms of the form $E_{k}\left(  2\rho\right)  $ and the third and fourth terms
a sum involving $E_{k}\left(  3\rho\right)  $. \ 

\bigskip

We now carry the process through in the most interesting cases, $n=2,3$. \ 

\bigskip\ 

\section{The case $k=0,n=2\label{n=2}$}

When $n=2$ we are interested in
\begin{align}
E_{1}\left(  \rho\right)   &  =\int_{\rho}^{\infty}\frac{1}{\tau}e^{-\tau
}d\tau\nonumber\\
&  =-e^{-\rho}\log\rho+\int_{\rho}^{\infty}e^{-\tau}\log\tau d\tau\nonumber\\
&  =-e^{-\rho}\log\rho+\int_{0}^{\infty}e^{-\tau}\log\tau d\tau-\int_{0}%
^{\rho}e^{-\tau}\log\tau d\tau\nonumber\\
&  =-e^{-\rho}\log\rho-\gamma-\rho\left(  \log\rho-1\right)  e^{-\rho
}+O\left(  \rho^{2}\log\rho\right)  ,\text{ for small }\rho,\nonumber\\
&  =-\log\rho-\gamma+\rho+O\left(  \rho^{2}\log\rho\right)  . \label{4.0}%
\end{align}

(See, for example, \cite{erd}, Chapter 1.) \ Also, for future purposes, using
$\left(  \ref{3.9a}\right)  $ we obtain%
\begin{align}
E_{2}\left(  \rho\right)   &  =\frac{e^{-\rho}}{\rho}-E_{1}\left(  \rho\right)
\label{4.01}\\
&  =\frac{1}{\rho}+\log\rho+\left(  \gamma-1\right)  -\frac{1}{2}\rho+O\left(
\rho^{2}\log\rho\right)  \text{ as }\rho\rightarrow0. \label{4.1}%
\end{align}

\bigskip

Looking now at $\left(  \ref{3.9}\right)  ,$ with $\rho=\varepsilon,$ we see
that as $\varepsilon\rightarrow0,$%
\[
C\log\varepsilon\rightarrow-1
\]
and%
\[
C=\frac{1}{\log\frac{1}{\varepsilon}}+O\left(  \frac{1}{\left(  \log\frac
{1}{\varepsilon}\right)  ^{2}}\right)  .
\]
Hence the series in $\left(  \ref{3.9}\right)  $ is in powers of $\frac
{1}{\log\frac{1}{\varepsilon}}.$ \ 

\bigskip

Also, we will work our approximations (in order to compare the results with
those of Hinch in \cite{hi}) to order $\frac{1}{\log^{2}\left(  \frac
{1}{\varepsilon}\right)  }$, \ so that (for example)
\[
u=\frac{a\left(  r\right)  }{\log\left(  \frac{1}{\varepsilon}\right)  }%
+\frac{b\left(  r\right)  }{\log^{2}\left(  \frac{1}{\varepsilon}\right)
}+O\left(  \log^{-3}\frac{1}{\varepsilon}\right)
\]
for any fixed value of $r$ \ ($\rho$ of order $\varepsilon$). \ This, as we
shall see, necessitates finding
\[
C=\frac{1}{\log\left(  \frac{1}{\varepsilon}\right)  }\left\{  1+\frac{A}%
{\log\left(  \frac{1}{\varepsilon}\right)  }+\frac{B}{\log^{2}\left(  \frac
{1}{\varepsilon}\right)  }+O\left(  \log^{-3}\left(  \frac{1}{\varepsilon
}\right)  \right)  \right\}  ,
\]
and requires use of all the terms in $\left(  \ref{3.9}\right)  .$ \ 

\bigskip

With this in mind, we look at the second term of $\left(  \ref{3.9}\right)  $.
\ Thus from $\left(  \ref{3.10}\right)  ,$%
\begin{equation}
\int_{\rho}^{\infty}E_{1}d\tau=-\rho E_{1}+e^{-\rho}, \label{4.2}%
\end{equation}
so that the second term is
\begin{align}
C^{2}\int_{\infty}^{\rho}\frac{1}{\tau}e^{-\tau}\left(  \tau E_{1}-e^{-\tau
}\right)  d\tau &  =C^{2}\left\{  \int_{\infty}^{\rho}e^{-\tau}E_{1}d\tau
-\int_{\infty}^{\rho}\frac{e^{-2\tau}}{\tau}d\tau\right\} \nonumber\\
&  =C^{2}\left\{  \left[  -e^{-\tau}E_{1}\right]  |_{\infty}^{\rho}%
-2\int_{\infty}^{\rho}\frac{e^{-2\tau}}{\tau}d\tau\right\} \nonumber\\
&  =C^{2}\left(  -e^{-\rho}E_{1}\left(  \rho\right)  +2E_{1}\left(
2\rho\right)  \right)  . \label{4.2a}%
\end{align}
From $\left(  \ref{4.0}\right)  $, the second term is therefore
\begin{align}
&  C^{2}\left(  \log\rho+\gamma-2\log2\rho-2\gamma+O\left(  \rho\right)
\right) \nonumber\\
&  =C^{2}\left(  -\log\rho-\gamma-2\log2+O\left(  \rho\right)  \right)
\label{4.3}%
\end{align}
as $\rho\rightarrow0$. \ \ 

\bigskip

In the third and fourth terms of $\left(  \ref{3.9}\right)  $ we need only the
leading terms, i.e. we can ignore the equivalent of $\gamma+2\log2$ \ in
$\left(  \ref{4.3}\right)  $. \ Using $\left(  \ref{4.2}\right)  $ the third
term becomes
\begin{equation}
\frac{1}{2}C^{3}\int_{\infty}^{\rho}\frac{e^{-\tau}}{\tau}\left(  e^{-\tau
}-\tau E_{1}\right)  ^{2}d\tau=\frac{1}{2}C^{3}\left(  \log\rho+O\left(
1\right)  \right)  \text{ as }\rho\rightarrow0. \label{4.4}%
\end{equation}

\bigskip

Finally, in the fourth term, the integrand in the $\tau$-integral is just the
second term, (as a function of $\sigma$), \ so that from $\left(
\ref{4.2}\right)  ,$ the fourth term is
\begin{equation}
M=-C^{3}\int_{\infty}^{\rho}\frac{e^{-\tau}}{\tau}\left[  \int_{\infty}^{\tau
}\left\{  -e^{-\sigma}E_{1}\left(  \sigma\right)  +2E_{1}\left(
2\sigma\right)  \right\}  d\sigma\right]  d\tau. \label{4.4a}%
\end{equation}

It is seen from $\left(  \ref{4.2}\right)  $ that for any $\tau\leq\infty,$
$\int_{0}^{\tau}E_{1}\left(  \sigma\right)  d\sigma$ \ converges. \ Hence we
can write the inner integral above in the form $\int_{\infty}^{0}+\int
_{0}^{\tau},$ \ and it follows that
\[
M=-C^{3}\int_{\infty}^{\rho}\frac{e^{-\tau}}{\tau}\left\{  K+r\left(
\tau\right)  \right\}  d\tau
\]
where $K$ is a constant, $r$ is bounded and $r\left(  \tau\right)  =O\left(
\tau\log\tau\right)  $ as $\tau\rightarrow0$. \ \ \ It further follows that
\[
M=C^{3}\left(  KE_{1}\left(  \rho\right)  +O\left(  1\right)  \right)  \text{
as }\rho\rightarrow0.
\]
We can evaluate $K$ using $\left(  \ref{4.2}\right)  $ and $\left(
\ref{4.0}\right)  $:%
\begin{align}
\int_{0}^{\infty}E_{1}\left(  2\sigma\right)  d\sigma &  =\frac{1}{2}\int
_{0}^{\infty}E_{1}\left(  u\right)  du=\frac{1}{2},\nonumber\\
\int_{0}^{\infty}e^{-\sigma}E_{1}\left(  \sigma\right)  d\sigma &  =\left[
-\left(  e^{-\sigma}-1\right)  E_{1}\right]  _{0}^{\infty}-\int_{0}^{\infty
}\left(  e^{-\sigma}-1\right)  \frac{e^{-\sigma}}{\sigma}d\sigma\nonumber\\
&  =\lim_{\sigma\rightarrow0}\left\{  -E_{1}\left(  2\sigma\right)
+E_{1}\left(  \sigma\right)  \right\}  =\lim_{\sigma\rightarrow0}\left(
\log2\sigma-\log\sigma\right)  =\log2. \label{4.4b}%
\end{align}
Hence, from $\left(  \ref{4.4a}\right)  $, \ the fourth term of $\left(
\ref{3.9}\right)  $ is
\begin{equation}
C^{3}\left\{  E_{1}\left(  \rho\right)  \left(  \log2-1\right)  +O\left(
1\right)  \right\}  =-C^{3}\left\{  \left(  \log2-1\right)  \log\rho+O\left(
1\right)  \right\}  \ \text{ as }\rho\rightarrow0. \label{4.5}%
\end{equation}

Now setting $\rho=\varepsilon$ \ and using $\left(  \ref{4.3}\right)  ,\left(
\ref{4.4}\right)  $, and $\left(  \ref{4.5}\right)  $, we obtain that
\begin{align*}
-1  &  =-C\left(  -\log\varepsilon-\gamma+O\left(  \varepsilon\right)
\right)  +C^{2}\left(  -\log\varepsilon-\gamma-2\log2+O\left(  \varepsilon
\right)  \right) \\
&  +\frac{1}{2}C^{3}\left(  \log\varepsilon+O\left(  1\right)  \right)
-C^{3}\left\{  \left(  \log2-1\right)  \log\varepsilon+O\left(  1\right)
\right\}
\end{align*}
as $\varepsilon\rightarrow0$. \ Hence,%
\begin{equation}
\frac{1}{\log\left(  \frac{1}{\varepsilon}\right)  }=C\left(  1-\frac{\gamma
}{\log\left(  \frac{1}{\varepsilon}\right)  }\right)  -C^{2}\left(
1-\frac{\gamma+2\log2}{\log\left(  \frac{1}{\varepsilon}\right)  }\right)
+C^{3}\left(  \frac{3}{2}-\log2\right)  +O\left(  \log^{-4}\left(  \frac
{1}{\varepsilon}\right)  \right)  ,
\end{equation}
and
\[
C=\frac{1}{\log\left(  \frac{1}{\varepsilon}\right)  }+\frac{A}{\log
^{2}\left(  \frac{1}{\varepsilon}\right)  }+\frac{B}{\log^{3}\left(  \frac
{1}{\varepsilon}\right)  }+O\left(  \frac{1}{\log^{4}\left(  \frac
{1}{\varepsilon}\right)  }\right)  ,
\]
where
\begin{align*}
-\gamma+A-1  &  =0,\\
B-\gamma A-2A+\left(  \gamma+2\log2\right)  +\frac{3}{2}-\log2  &  =0.
\end{align*}
Hence,%
\begin{align*}
A  &  =\gamma+1\\
B  &  =\gamma^{2}+2\gamma+\frac{1}{2}-\log2.
\end{align*}
Thus, for fixed $r,$ $\rho$ of order $\varepsilon,$ we have, with
$\lambda=\log\left(  \frac{1}{\varepsilon}\right)  ,$%
\begin{align*}
u-1  &  =\left(  \frac{1}{\lambda}+\frac{\gamma+1}{\lambda^{2}}+\frac{\left(
\gamma+1\right)  ^{2}-\frac{1}{2}-\log2}{\lambda^{3}}\right)  \left(  \log
r+\log\varepsilon+\gamma\right) \\
&  +\frac{1}{\lambda^{2}}\left(  1+\frac{2\left(  \gamma+1\right)  }{\lambda
}\right)  \left(  -\log r-\log\varepsilon-\gamma-2\log2\right) \\
&  +\frac{1}{\lambda^{3}}\left(  \frac{3}{2}-\log2\right)  \left(  \log
r+\log\varepsilon\right)  +O\left(  \lambda^{-4}\right)  ,
\end{align*}
so that, after cancellation,%
\[
u=\frac{\log r}{\lambda}+\frac{\gamma\log r}{\lambda^{2}}+O\left(
\lambda^{-3}\right)  .
\]
This is the \textquotedblleft inner expansion\textquotedblright. \ For the
\textquotedblleft outer expansion\textquotedblright, i.e. fixed $\rho,$ \ $r$
of order $\frac{1}{\varepsilon},$ we use $\left(  \ref{3.9}\right)  ,$
truncated to second order, to get%
\[
u-1=-E_{1}\left(  \rho\right)  \left(  \frac{1}{\lambda}+\frac{\gamma
+1}{\lambda^{2}}\right)  +\frac{1}{\lambda^{2}}\left(  2E_{1}\left(
2\rho\right)  -e^{-\rho}E_{1}\left(  \rho\right)  \right)  +O\left(
\lambda^{-3}\right)  .
\]
These results are in accordance with those of Hinch and of others on this
problem. \ 

\bigskip

\section{The case $k=0,n=3\label{n=3}$}

Here we are interested in (from $\left(  \ref{4.0}\right)  $ and $\left(
\ref{4.01}\right)  $)%
\[
E_{2}\left(  \rho\right)  =\frac{e^{-\rho}}{\rho}-E_{1}\left(  \rho\right)
=\frac{1}{\rho}+\log\rho+\left(  \gamma-1\right)  -\frac{1}{2}\rho+O\left(
\rho^{2}\log\rho\right)  \text{ as \ }\rho\rightarrow0.
\]
Thus, the first term on the right of $\left(  \ref{3.9}\right)  $ evaluated at
$\rho=\varepsilon$ \ is
\[
-C\left(  1+\varepsilon\log\varepsilon+\left(  \gamma-1\right)  +O\left(
\varepsilon^{2}\right)  \right)  \text{ as }\varepsilon\rightarrow0.
\]
The second term is
\begin{align*}
&  \left(  C\varepsilon\right)  ^{2}\int_{\infty}^{\rho}\frac{1}{\tau^{2}%
}e^{-\tau}\left(  \int_{\infty}^{\tau}E_{2}d\sigma\right)  d\tau\\
&  =\left(  C\varepsilon\right)  ^{2}\left\{  \left[  -E_{2}\left(
\tau\right)  \int_{\infty}^{\tau}E_{2}\left(  \sigma\right)  d\sigma\right]
_{\infty}^{\rho}+\int_{\infty}^{\rho}E_{2}^{2}d\tau\right\} \\
&  =\left(  C\varepsilon\right)  ^{2}\left\{  -E_{2}\left(  \rho\right)
\int_{\infty}^{\rho}E_{2}\left(  \tau\right)  d\tau+\int_{\infty}^{\rho}%
E_{2}^{2}d\tau\right\}  .
\end{align*}
From $\left(  \ref{4.01}\right)  $ we see that
\[
\int_{\infty}^{\rho}E_{2}^{2}d\tau=-\frac{1}{\rho}+\log^{2}\rho+O\left(
\log\rho\right)  \text{ as }\rho\rightarrow0,
\]
while from $\left(  \ref{3.10}\right)  ,$
\begin{align*}
\int_{\infty}^{\rho}E_{2}d\tau &  =\rho E_{2}-E_{1}=1+\log\rho+\gamma+O\left(
\rho\log\rho\right)  ,\\
E_{2}\int_{\infty}^{\rho}E_{2}d\tau &  =\frac{1}{\rho}\log\rho+\frac{\gamma
+1}{\rho}+O\left(  \log^{2}\rho\right)  .
\end{align*}
In all, the second term is
\[
\left(  C\varepsilon\right)  ^{2}\left\{  -\frac{1}{\rho}\log\rho-\frac
{\gamma+2}{\rho}+O\left(  \log^{2}\rho\right)  \right\}  .
\]

It is readily verified that the third and fourth terms in $\left(
\ref{3.9}\right)  $give $O\left\{  C^{3}\varepsilon^{3}\left(  \frac{1}{\rho
}\log^{2}\rho\right)  \right\}  ,$ which is negligible. Thus, evaluating
$\left(  \ref{3.9}\right)  $ at $\rho=\varepsilon$, \ we have
\[
-1=-C\varepsilon\left(  \frac{1}{\varepsilon}+\log\varepsilon+\gamma-1\right)
+\left(  C\varepsilon\right)  ^{2}\left(  -\frac{1}{\varepsilon}%
\log\varepsilon-\frac{\gamma+2}{\varepsilon}\right)  +O\left(  C^{3}%
\varepsilon^{2}\log^{2}\varepsilon\right)  ,
\]
so that
\[
C=1-2\varepsilon\log\varepsilon-\varepsilon\left(  2\gamma+1\right)  +O\left(
\varepsilon^{2}\log^{2}\varepsilon\right)  .
\]
\bigskip

Then, for fixed $r,$ $\rho$ of order $\varepsilon,$ we have%

\begin{align*}
u-1  &  =-\varepsilon\left(  1-2\varepsilon\log\varepsilon-\varepsilon\left(
2\gamma+1\right)  \right)  \left(  \frac{1}{\varepsilon r}+\log\varepsilon
+\log r+\gamma-1\right) \\
&  +\varepsilon^{2}\left(  -\frac{1}{\varepsilon r}\left(  \log\varepsilon
+\log r\right)  -\frac{\gamma+2}{\varepsilon r}\right)  +O\left(
\varepsilon^{2}\log^{2}\varepsilon\right)  ,
\end{align*}%
\begin{align*}
u  &  =1-\frac{1}{r}-\varepsilon\log\varepsilon\left(  1-\frac{1}{r}\right)
-\varepsilon\left(  \log r+\frac{\log r}{r}\right)  +\varepsilon\left(
1-\gamma\right)  \left(  1-\frac{1}{r}\right) \\
&  +O\left(  \varepsilon^{2}\log^{2}\varepsilon\right)  .
\end{align*}

For fixed $\rho,$ \ $r$ of order $\varepsilon^{-1},$ we again use $\left(
\ref{3.9}\right)  $, to give
\begin{align}
u-1  &  =-\varepsilon\left(  1-2\varepsilon\log\varepsilon-\varepsilon\left(
2\gamma+1\right)  \right)  E_{2}\left(  \rho\right) \nonumber\\
&  +\varepsilon^{2}\left\{  E_{1}\left(  \rho\right)  E_{2}\left(
\rho\right)  -\rho E_{2}^{2}\left(  \rho\right)  -\int_{\rho}^{\infty}%
E_{2}^{2}d\tau\right\}  +O\left(  \varepsilon^{3}\right)  . \label{5.1}%
\end{align}
Again, these results are in agreement with those of Hinch, and others,
although $\left(  \ref{5.1}\right)  $ gives one term further.

\bigskip

Remark 3. \ It is of interest to consider what happens when $n<2,$\ since, at
least for $n\geq1$, \ there still exists a unique solution. \ The equation
$\left(  \ref{3.9}\right)  $\ is still valid at $\rho=\varepsilon$, \ but
since $E_{n-1}\left(  \rho\right)  $\ is no longer singular at $\rho=0$\ for
$n<2$, \ \ $\left(  \ref{3.9}\right)  $\ with $\rho=\varepsilon$\ becomes
merely an implicit equation for $C\varepsilon^{n-2}$. \ This tells us that
$C\rightarrow0$, \ since $\varepsilon^{n-2}\rightarrow\infty$, \ but we no
longer get an asymptotic expansion. \ In particular, \ it is no longer obvious
that $C$\ is unique. \ Of course, we know this from Theorem \ref{th1} if
$n\geq1.$\ 

\section{The case $k=1$}

We can in fact treat a generalization which causes no further difficulties,%
\begin{equation}
u^{\prime\prime}+\frac{n-1}{r}u^{\prime}+f\left(  u\right)  u^{\prime
2}+\varepsilon u~u^{\prime}=0, \label{6.2}%
\end{equation}
with the same boundary conditions. \ As before, we will compare our results
with those of Hinch in \cite{hi}.

\bigskip

As remarked in the proof of Theorem 1, the solution will necessarily have
$u^{\prime}>0$ so that conditions on $f\left(  u\right)  $ are necessary only
for $0\leq u\leq1$. \ We require only that $f$ be continuous and positive in
this interval.

\bigskip

Then $\left(  \ref{6.2}\right)  $ can be written as
\[
\frac{\left(  r^{n-1}u^{\prime}\right)  ^{\prime}}{r^{n-1}u^{\prime}}+f\left(
u\right)  u^{\prime}+\varepsilon u=0,
\]
so that%
\[
\log\left(  r^{n-1}u^{\prime}\right)  =-F\left(  u\right)  -\varepsilon
\int_{1}^{r}udt+A
\]
for some constant $A,$ where
\[
F\left(  u\right)  =\int_{0}^{u}f\left(  s\right)  ds.
\]
This becomes
\[
e^{F\left(  u\right)  }u^{\prime}=\frac{C}{r^{n-1}}e^{-\varepsilon
r-\varepsilon\int_{\infty}^{r}\left(  u-1\right)  ds},
\]
or, on integration,%
\[
G\left(  u\right)  -G\left(  1\right)  =C\int_{\infty}^{r}\frac{1}{t^{n-1}%
}e^{-\varepsilon t-\varepsilon\int_{\infty}^{t}\left(  u-1\right)  ds}dt,
\]
where%
\[
G\left(  u\right)  =\int_{0}^{u}e^{F\left(  v\right)  }dv.
\]

In order to keep the manipulations simple and effect comparisons, we will
consider from here the Lagerstrom model, \ where $f\left(  u\right)  =1,$
$F\left(  u\right)  =u,$ \ $G\left(  u\right)  =e^{u}-1.$ Then, with
$\varepsilon r=\rho$, \ $\varepsilon t=\tau$, \ we have
\begin{equation}
e^{u}-e=C\varepsilon^{n-2}\int_{\infty}^{\rho}\frac{1}{\tau^{n-1}}e^{-\tau
}e^{-\int_{\infty}^{\tau}\left(  u-1\right)  d\sigma}d\tau, \label{6.2a}%
\end{equation}
and writing
\[
u-1=\frac{u-1}{e^{u}-e}\left(  e^{u}-e\right)  ,
\]
we get
\begin{equation}
e^{u}-e=C\varepsilon^{n-2}\int_{\infty}^{\rho}\frac{1}{\tau^{n-1}}e^{-\tau
}e^{-\int_{\infty}^{\tau}\frac{u-1}{e^{u}-e}\left(  e^{u}-e\right)  d\sigma
}d\tau. \label{6.3}%
\end{equation}

\ As in section \ref{infser}, we can integrate by parts, and since $0\leq
\frac{u-1}{e^{u}-e}\leq1$ in $0\leq u<1,$ we will develop a convergent series
as before. \ To get the first three terms (necessary to give Hinch's accuracy
when $n=2$), \ we have from $\left(  \ref{6.3}\right)  $ that
\begin{align}
&  e^{u}-e\nonumber\\
&  =C\varepsilon^{n-2}\int_{\infty}^{\rho}\frac{1}{\tau^{n-1}}e^{-\tau
}\left\{  1-\int_{\infty}^{\tau}\frac{u-1}{e^{u}-e}\left(  e^{u}-e\right)
d\sigma+\frac{1}{2}\left(  \int_{\infty}^{\tau}\frac{u-1}{e^{u}-e}\left(
e^{u}-e\right)  d\sigma\right)  ^{2}+\cdot\cdot\cdot\right\}  d\tau.
\label{6.4a}%
\end{align}

As before, since $e^{u}-e\rightarrow0$ \ exponentially fast as $\rho
\rightarrow\infty,$ \ the series in the integrand converges uniformly for
large $\tau$, \ so that $\left(  \ref{6.4a}\right)  $ is valid for large
$\rho.$ \ But again we need to extend it down to $\rho=\varepsilon$. \ From
$\left(  \ref{6.2a}\right)  $we have
\[
e^{u}-e\leq C\varepsilon^{n-2}E_{n-1}\left(  \rho\right)
\]
and so the convergence proof is the same as that preceding $\left(
\ref{3.9}\right)  .$ \ 

\bigskip

Before proceeding further with $n=2$, \ we make a couple of remarks about the
simpler case $n>2$. \ Then, as we saw in subsection \ref{n=3}, \ only two
terms are necessary to give the required accuracy, \ and then $\left(
\ref{6.4a}\right)  $ gives
\[
e^{u}-u=C\varepsilon^{n-2}\int_{\infty}^{\rho}\frac{e^{-\tau}}{\tau^{n-1}%
}\left\{  1-\int_{\infty}^{\tau}\frac{u-1}{e^{u}-e}\left(  e^{u}-e\right)
d\sigma+\cdot\cdot\cdot\right\}
\]
and since $\frac{u-1}{e^{u}-e}$ appears in what is already the highest order
term, we can replace it by its limit as $u\rightarrow1,$ i.e. $\frac{1}{e}.$
\ Thus we get, to the required order,
\[
e^{u}-e=-C\varepsilon^{n-2}E_{n-1}-\left(  \frac{C\varepsilon^{n-2}}%
{e}\right)  \int_{\infty}^{\rho}\frac{e^{-\tau}}{\tau^{n-1}}\int_{\infty
}^{\tau}\left(  e^{u}-e\right)  d\sigma d\tau.
\]
(We will proceed more carefully for $n=2$.) \ This, apart from the factor
$\frac{1}{e},$ is the same equation as we dealt with in section \ref{n=3}
(with $e^{u}-e$ in place of $u-1$), and the solution can be written down from
there. \ (If we had a general function $f$ in place of $1,$ we would get
\[
e^{F\left(  u\right)  }-e^{F\left(  1\right)  }=-C\varepsilon^{n-2}%
E_{n-1}-\frac{C\varepsilon^{n-2}}{e^{F\left(  1\right)  }f\left(  1\right)
}\int_{\infty}^{\rho}\frac{e^{-\tau}}{\tau^{n-1}}\left(  \int_{\infty}^{\tau
}\left(  e^{F\left(  u\right)  }-e^{F\left(  1\right)  }\right)
d\sigma\right)  d\tau.\text{)}%
\]

\bigskip

Turning now to the case $n=2,$ and $F\left(  u\right)  =u,$ we need three
terms on the right of $\left(  \ref{6.4a}\right)  .$ Thus,
\begin{equation}
\frac{u-1}{e^{u}-e}=\frac{1}{e}-\frac{1}{2e^{2}}\left(  e^{u}-e\right)
+O\left(  e^{u}-e\right)  ^{2}\text{ as }u\rightarrow1.
\end{equation}
We follow the method used just before $\left(  \ref{3.9}\right)  $ and obtain
from $\left(  \ref{6.4a}\right)  $ that
\begin{align*}
e^{u}-e  &  =-C\varepsilon^{n-2}E_{n-1}+\frac{1}{e}\left(  C\varepsilon
^{n-2}\right)  ^{2}\int_{\infty}^{\rho}\frac{1}{\tau^{n-1}}e^{-\tau}\left(
\int_{\infty}^{\tau}E_{n-1}d\sigma\right)  d\tau\\
&  +\frac{1}{2e^{2}}\left(  C\varepsilon^{n-2}\right)  ^{3}\int_{\infty}%
^{\rho}\frac{1}{\tau^{n-1}}e^{-\tau}\left(  \int_{\infty}^{\tau}E_{n-1}%
^{2}d\sigma\right)  d\tau\\
&  +\frac{1}{2e^{2}}\left(  C\varepsilon^{n-2}\right)  ^{3}\int_{\infty}%
^{\rho}\frac{1}{\tau^{n-1}}e^{-\tau}\left(  \int_{\infty}^{\tau}E_{n-1}%
d\sigma\right)  ^{2}d\tau\\
&  -\frac{1}{e^{2}}\left(  C\varepsilon^{n-2}\right)  ^{3}\int_{\infty}^{\rho
}\frac{1}{\tau^{n-1}}e^{-\tau}\int_{\infty}^{\tau}\left\{  \int_{\infty
}^{\sigma}\frac{1}{s^{n-1}}e^{-s}\left(  \int_{\infty}^{s}E_{n-1}dt\right)
ds\right\}  d\sigma d\tau+O\left(  \Phi^{4}\right) \\
&  =-C\varepsilon^{n-2}E_{n-1}+F_{1}+F_{2}+F_{3}+F_{4}+O\left(  \Phi
^{4}\right)  ,\text{ \ say.}%
\end{align*}
As before, if $n>2$ then this is valid for any $C$ as $\varepsilon
\rightarrow0,$ uniformly in $\rho\geq\varepsilon,$\ while if $n=2,$ \ it is
valid as $C\rightarrow0.$ \ 

For $n=2$ we can continue to follow the argument in section 4. \ Thus, as
$\rho\rightarrow0,$
\begin{align*}
F_{1}  &  =\frac{1}{e}C^{2}\left(  -\log\rho-\gamma-2\log2+O\left(
\rho\right)  \right)  \text{ }\\
F_{3}  &  =\frac{1}{2e^{2}}C^{3}\left(  \log\rho+O\left(  1\right)  \right) \\
F_{4}  &  =-\frac{1}{e^{2}}C^{3}\left[  \left(  \log2-1\right)  \log
\rho+O\left(  1\right)  \right]  .
\end{align*}
The term $F_{2}$ did not appear before. \ \ Only the highest order term is
needed for our expansion and this is
\[
-\frac{1}{2e^{2}}C^{3}\left(  \int_{\infty}^{\rho}\frac{1}{\tau}e^{-\tau}%
d\tau\right)  \int_{0}^{\infty}E_{1}^{2}d\sigma.
\]
Now
\begin{align*}
\int_{0}^{\infty}E_{1}^{2}d\sigma &  =\left[  \tau E_{1}^{2}\right]
_{0}^{\infty}+2\int_{0}^{\infty}\tau\frac{e^{-\tau}}{\tau}E_{1}d\tau\\
&  =2\int_{0}^{\infty}e^{-\tau}E_{1}d\tau=2\log2,\text{ from }\left(
\ref{4.4b}\right)  .
\end{align*}
Thus,%
\[
F_{2}=\frac{1}{e^{2}}C^{3}E_{1}\left(  \log2+O\left(  \rho\log^{2}\rho\right)
\right)  =-\frac{1}{e^{2}}C^{3}\left(  \left(  \log2\right)  \log\rho+O\left(
1\right)  \right)  ,
\]
and, evaluating at $\rho=\varepsilon,$ we have
\begin{align*}
1-e  &  =C\left(  \log\varepsilon+\gamma+O\left(  \varepsilon\right)  \right)
\\
&  -\frac{1}{e}C^{2}\left(  \log\varepsilon+\gamma+2\log2+O\left(
\varepsilon\right)  \right)  +\frac{1}{2e^{2}}C^{3}\log\varepsilon\left(
-2\log2+1-2\log2+2\right)  +O\left(  C^{3}\right)  ,\\
\frac{e-1}{\log\frac{1}{\varepsilon}}  &  =C\left(  1-\frac{\gamma}%
{\log\left(  \frac{1}{\varepsilon}\right)  }\right)  -\frac{1}{e}C^{2}\left(
1-\frac{\gamma+2\log2}{\log\left(  \frac{1}{\varepsilon}\right)  }\right) \\
&  +\frac{1}{2e^{2}}C^{3}\left(  3-4\log2+O\left(  \frac{C\varepsilon}%
{\log\varepsilon}\right)  +O\left(  \frac{C^{2}\varepsilon}{\log\varepsilon
}\right)  +O\left(  \frac{C^{3}}{\log\varepsilon}\right)  \right)  .
\end{align*}
Hence if
\[
C=\frac{e-1}{\log\left(  \frac{1}{\varepsilon}\right)  }+\frac{A}{\log
^{2}\left(  \frac{1}{\varepsilon}\right)  }+\frac{B}{\log^{3}\left(  \frac
{1}{\varepsilon}\right)  }+O\left(  \log^{-4}\left(  \frac{1}{\varepsilon
}\right)  \right)  ,
\]
then
\begin{align*}
-\gamma\left(  e-1\right)  +A-\frac{\left(  e-1\right)  ^{2}}{e}  &  =0,\\
A  &  =\frac{e-1}{e}\left(  \gamma e+e-1\right)  ,\\
B-A\gamma+\frac{\left(  e-1\right)  ^{2}}{e}\left(  \gamma+2\log2\right)
-\frac{2A\left(  e-1\right)  }{e}+  &  \frac{1}{2e^{2}}\left(  e-1\right)
^{3}\left(  3-4\log2\right)  =0.
\end{align*}
We can of course calculate $B,$ but in fact its value will be be irrelevant to
the level of approximation that we take.

\bigskip

Then, for fixed $r$ \ $\ $($\rho$ of order $\varepsilon$ ), we have, with
$l=\log\left(  \frac{1}{\varepsilon}\right)  ,$%
\begin{align}
e^{u}-e  &  =\left(  e-1\right)  \left\{  \frac{1}{l}+\frac{\gamma+1-\frac
{1}{e}}{l^{2}}+\frac{B/\left(  e-1\right)  }{l^{3}}\right\}  \left(
\log\varepsilon+\log r+\gamma\right) \nonumber\\
&  +\frac{1}{e}\left(  e-1\right)  ^{2}\left\{  \frac{1}{l^{2}}+\frac{2\left(
\gamma+1-\frac{1}{e}\right)  }{l^{3}}\right\}  \left(  -\log\varepsilon-\log
r-\gamma-2\log2\right) \nonumber\\
&  +\frac{1}{2e^{2}}\frac{\left(  e-1\right)  ^{3}}{l^{3}}\left(
3-4\log2\right)  \left(  \log\varepsilon+\log r\right)  +O\left(
l^{-3}\right) \nonumber\\
&  =1-e+\frac{\left(  e-1\right)  \log r}{l}+\frac{\gamma\left(  e-1\right)
\log r}{l^{2}}+O\left(  l^{-3}\right)  . \label{6.5}%
\end{align}

(Note that the definitions of $A$ and $B$ were such that $u=0$ \ at $r=1$ up
to and including order $l^{-2},$ so that to that order there can be only terms
in $\log r,$ not constant terms. We do not need the explicit value of $B$.) To
obtain $u,$ we have to invert, so that
\[
u=\log\left\{  1+\frac{e-1}{l}\log r+\frac{\gamma\left(  e-1\right)  }{l^{2}%
}\log r+O\left(  l^{-3}\right)  \right\}  .
\]

\bigskip

For fixed $\rho,$ $r$ of order $\varepsilon^{-1},$ \ we have
\[
e^{u}-e=-\frac{e-1}{l}\left(  1+\frac{\gamma+1-\frac{1}{e}}{l}\right)
E_{1}\left(  \rho\right)  +\frac{\left(  e-1\right)  ^{2}}{el^{2}}\left(
2E_{1}\left(  2\rho\right)  -e^{-\rho}E_{1}\left(  \rho\right)  \right)
+O\left(  l^{-3}\right)  .
\]
Thus%
\begin{align}
u-1  &  =\frac{1}{e}\left(  e^{u}-e\right)  -\frac{1}{2e^{2}}\left(
e^{u}-e\right)  ^{2}+\cdot\cdot\cdot\nonumber\\
&  =-\frac{e-1}{e}\left(  1+\frac{\gamma+1-\frac{1}{e}}{l}\right)  \frac
{E_{1}\left(  \rho\right)  }{l}+\frac{\left(  e-1\right)  ^{2}}{e^{2}}%
\frac{\left(  2E_{1}\left(  2\rho\right)  -e^{-\rho}E_{1}\left(  \rho\right)
\right)  }{l^{2}}\nonumber\\
&  -\frac{\left(  e-1\right)  ^{2}}{2e^{2}}\frac{E_{1}^{2}\left(  \rho\right)
}{l^{2}}+O\left(  l^{-3}\right)  . \label{6.6}%
\end{align}
Again, these results are consistent with those of Hinch and others, except
that Hinch has an algebraic mistake which in $\left(  \ref{6.6}\right)  $
replaces $\gamma+1-\frac{1}{e}$by \ $\gamma-1+\frac{1}{e}$. \ \ 

\section{Final Remarks}

Starting with Lagerstrom, the terms involving $\log\varepsilon$ \ in the inner
expansions have been considered difficult to explain. \ They are often called
\textquotedblleft switchback\textquotedblright\ terms, because there is
nothing obvious in the equation which indicates the need for such terms, and
because, starting with an expansion in powers of $\varepsilon$, \ one finds
inconsistent results which are only resolved by adding terms of lower order,
that is, powers of $\varepsilon\log\varepsilon$. \ The recent approach to the
problem by geometric perturbation theory explains this by reference to a
\textquotedblleft resonance phenomenon\textquotedblright, which is too
complicated for us to describe here \cite{pop1},\cite{pop2}.

\bigskip

In our work, the necessity for such terms is seen already from the equation
$\left(  \ref{3.2}\right)  $ and the resulting expansion $\left(
\ref{3.3a}\right)  :$\
\[
u\left(  \rho\right)  -1=C\varepsilon^{n-2}\int_{\infty}^{\rho}\frac{1}%
{\tau^{n-1}}e^{-\tau}\left\{  1-\int_{\infty}^{\tau}\left(  u-1\right)
d\sigma+\frac{1}{2}\left(  \int_{\infty}^{\tau}\left(  u-1\right)
d\sigma\right)  ^{2}-\cdot\cdot\cdot\right\}  d\tau.
\]
In the existence proof it was seen in $\left(  \ref{e4a}\right)  $ that
$C=O\left(  1\right)  $ as $\varepsilon\rightarrow0$. \ On the right of
$\left(  \ref{3.3a}\right)  $ the first term is simply $-C\varepsilon
^{n-2}E_{n-1}\left(  \rho\right)  ,$ and the simple expansions given for
$E_{1}$ and $E_{2}$ show immediately the need for the logarithmic terms.
\ There is no \textquotedblleft switchback\textquotedblright, because the
procedure does not start with any assumption about the nature of the
expansion, and there is no need for a "matching".

\bigskip

A number of authors have noted that the outer expansion is a uniformly valid
asymptotic expansion on $[1,\infty),$ and therefore it \textquotedblleft
contains\textquotedblright\ the inner expansion \cite{hunter}, though this is
more subtle when $k=1$ \cite{skinner}. \ Our twist on this is that both
expansions are contained in the uniformly convergent series defined implicitly
by $\left(  \ref{3.3a}\right)  .$ \ The simple derivation of this series, via
the integral equation $\left(  \ref{3.2}\right)  $ is new, as far as we know.

\bigskip

\bigskip

\begin{thebibliography}{99}                                                                                               %


\bibitem {cohen}Cohen, D. S., A. Fokas and P. Lagerstrom, Proof of some
asymptotic results for a model equation for low Reynolds number flow, SIAM J.
Appld. Math. 35 (1978), 187-207.

\bibitem {erd}Erdelyi, A. et. al., Higher Transcendental Functions, Vol. 1,
\ McGraw Hill, 1953.

\bibitem {hi}Hinch, E. J., Perturbation Methods, Cambridge University Press, 1991.

\bibitem {hsiao}Hsiao, G. C., Singular perturbations for a nonlinear
differential equation with a small parameter, Siam J. Math Anal. 4 (1973), 283-301.

\bibitem {hunter}Hunter, C., M. Tajdari and S. D. Boyer, On Lagerstrom's model
of slow incompressible viscous flow, Siam J. Appld. Math.50 1990, 48-63.

\bibitem {kla}Kaplun, S. and P. A. Lagerstrom, Asymptotic expansions of
Navier-Stokes solutions for small Reynolds number, J. Math. Mech. 6 (1957),
585-593. \ 

\bibitem {la}Lagerstrom, P. A. and R. G. Casten, Basic Concepts in Singular
Perturbation Techniques, Siam Review 14 (1972), 63-120.

\bibitem {lagerrein}Lagerstrom, P.A. and Reinelt, C. A., Note on logarithmic
and switchback terms in regular and singular perturbation expansions, Siam
Review 44 (1984), 451-462.

\bibitem {macgillivray}MacGillivray, A. D., On a model equation of Lagerstrom,
Siam J. Appld. Math. 34 (1978), 804-812.

\bibitem {pop1}Popovic, N. and P. Szmolyan, A geometric analysis of the
Lagerstrom model problem, J. Differential Equations 199 (2004), 290-325.

\bibitem {pop2}Popovic, N. and P. Szmolyan, Rigorous asymptotic expansions for
Lagerstrom's model equations, a geometric approach, Nonlinear Analysis 59
(2004), 531-565.

\bibitem {rosenblat}Rosenblat, S. and J. Shepherd, On the asymptotic solution
of the Lagerstrom model equation, Siam J. Appld. Math. 29 (1975), 110-120.

\bibitem {skinner}Skinner, L. A., Note on the Lagerstrom singular perturbation
models, Siam J. Appld. Math. 41 (1981), 362-364.

\bibitem {tam}Tam, K. , On the Lagerstrom model for flow at low Reynolds
numbers, J. Math. Anal. Appl.., 49 (1975), 286-294.
\end{thebibliography}
\end{document}